\documentclass[a4paper, 11pt]{amsart} 

\usepackage[svgnames]{xcolor} 

\usepackage[utf8]{inputenc} 
\usepackage[T1]{fontenc} 
\usepackage{PTSerif} 

\usepackage[linktocpage]{hyperref} 
\usepackage{amsmath}
\usepackage{amssymb}
\usepackage{amsthm}
\usepackage{blindtext}
\usepackage[margin = 1 in]{geometry}
\usepackage{mathtools}
\usepackage{graphicx}
\usepackage{tikz-cd}
\usepackage{footnote}
\usepackage{fancyhdr}
\usepackage{hyperref}
\usepackage{stmaryrd}
\usepackage{comment}
\usepackage{mathrsfs}

\usepackage{biblatex}
\usepackage{todonotes}
\makeatletter
\providecommand\@dotsep{5}
\renewcommand{\listoftodos}[1][\@todonotes@todolistname]{%
  \@starttoc{tdo}{#1}}
\makeatother

\usepackage[normalem]{ulem}
\hypersetup{
    colorlinks,
    linkcolor=blue,
    urlcolor=blue
}

\newtheorem{theorem}{Theorem}[section]
\newtheorem{corollary}[theorem]{Corollary}
\newtheorem{lemma}[theorem]{Lemma}
\newtheorem{proposition}[theorem]{Proposition}

\theoremstyle{definition}

\theoremstyle{remark}
\newtheorem*{remark}{Remark}

\graphicspath{ {images/} }

\newcommand{\mb}{\mathbb}

\newcommand{\mf}{\mathfrak}

\newcommand{\Q}{\mb{Q}}
\newcommand{\R}{\mb{R}}
\newcommand{\Z}{\mb{Z}}
\newcommand{\C}{\mb{C}}

\renewcommand{\P}{\mb{P}}

\newcommand{\eps}{\epsilon}
\newcommand*{\sheafhom}{\mathcal{H}\kern -.5pt om}

\DeclareMathOperator{\Gal}{Gal}

\DeclareMathOperator{\Aut}{Aut}

\DeclareMathOperator{\Sp}{Sp}
\DeclareMathOperator{\Ort}{O}
\DeclareMathOperator{\prim}{prim}
\DeclareMathOperator{\Sym}{Sym}

\DeclareMathOperator{\Spo}{SpO}
\DeclareMathOperator{\parity}{parity}
\DeclareMathOperator{\Var}{Var}

\title{Effectivity for bounds on points of good reduction in the moduli space of hypersurfaces}
\author{Caleb Ji}
\date{\today}

\begin{document}
\maketitle 
\begin{abstract}
    Let $S$ be a finite set of primes.  For sufficiently large $n$ and $d$, Lawrence and Venkatesh proved that in the moduli space of hypersurfaces of degree $d$ in $\P^n$, the locus of points with good reduction outside $S$ is not Zariski dense.  We make this result effective by computing explicit values of $n$ and $d$ for which this statement holds.  We accomplish this by giving a more precise computation and analysis of the Hodge numbers of these hypersurfaces and check that they satisfy certain bounds. 
\end{abstract}

\section{Introduction} 
Given a family of varieties over a number field, a Shafarevich-type conjecture states that there are only finitely many with good reduction outside some finite set of primes.  The original versions, for abelian varieties of fixed dimension with principal polarization and for curves with fixed genus, were proven by Faltings \cite{Faltings} in his proof of the Mordell conjecture.  

Recently, Lawrence-Venkatesh \cite{LV} initiated a new approach toward these types of questions using p-adic Hodge theory and monodromy calculations.  This approach led to a new proof of the Mordell conjecture, and while it did not prove the original version of the Shafarevich conjecture, it is applicable to other versions of the Shafarevich conjecture.  In particular, they proved the following results. 

\begin{theorem}
    \cite[Theorem 10.1]{LV}
    \label{thm: 10.1}
    Let $\pi\colon X\rightarrow Y$ be a smooth proper morphism over $\Z[S^{-1}]$, whose fibers are geometrically connected of relative dimension $n-1$.  Suppose that
the monodromy representation has large image, and moreover that 
\begin{equation}
    \label{eqn: 1} 
    \sum_{p>0} h^p\ge h^0+\dim(Y)
\end{equation}
and 
\begin{equation}
    \label{eqn: 2} 
    \sum_{p>0} ph^p\ge T(h^0+\dim(Y)) + T\left( \frac{3}{2}h^0 + \dim(Y)\right). 
\end{equation} 
Then $Y(\Z[S^{-1}])$ is not Zariski dense in in $Y$. 
\end{theorem}

We will explain the meaning of these conditions in Section \ref{sec: background}.  In the moduli space of hypersurfaces in $\P^n$ with fixed degree $d$, the monodromy condition is generally satisfied, so that the following result implies Zariski non-density. 

\begin{proposition}
\label{prop: 10.2}
\cite[Proposition 10.2]{LV}
     There exists $n_0$ and a function $D_0(n)$ such that both (\ref{eqn: 1}) and (\ref{eqn: 2})
apply to $X\rightarrow Y$ the universal family of hypersurfaces in $\P^n$ of degree $d$, so long as
$n > n_0$ and $d > D_0(n)$. 
\end{proposition} 

While this result only gives Zariski non-density rather than finiteness, it is plausible that in some scenarios it can be iterated to reach finiteness; see \cite{LS} for an instance where this idea is successfully carried out.  

The goal of this paper is to make Proposition \ref{prop: 10.2} effective.  We prove the following results. 

\begin{theorem}
\label{thm: effective}
For all $n > 2000$, there exists some degree $d(n)$ such that the hypersurfaces of degree $d(n)$ in $\P^n$ with good reduction outside $S$ are not Zariski dense in their moduli space.  [Proposition \ref{prop: n eff}]

Furthermore, for all $d\ge n>500000$, the hypersurfaces of
degree $d$ in $\P^n$ with good reduction outside $S$ are not Zariski dense in their moduli space. [Proposition \ref{prop: n d effective}] 
\end{theorem}  

The proof of Proposition \ref{prop: n eff}, which makes $n$ effective, follows the same lines as the proof of Proposition 10.2 of \cite{LV} except that explicit estimates are used in place of asymptotic results in some places.  However, we still use the convergence of the distribution of Hodge numbers to a distribution defined by Eulerian numbers as $d$ goes to infinity.  To make $d$ effective, in the proof of Proposition \ref{prop: n d effective} we work with the Hodge numbers directly.  The bounds obtained here are not optimized.  Often, we have chosen to use weaker bounds that simplify the analysis in order to illustrate the general method more clearly.

In \cite{Baldi}, the authors proved a case of the refined Bombieri-Lang conjecture whenever the hypotheses of Proposition \ref{prop: 10.2} hold.  This statement is now made effective through Theorem \ref{thm: effective}. 

The outline of this paper is as follows. 
In Section 2, we provide background on this approach and explain how it can be reduced to the verification of numerical conditions.  In Section 3 we prove explicit bounds regarding Eulerian numbers that allow us to make the dimension effective.  In Section 4, we compute the Hodge numbers of hypersurfaces and their statistics directly, which allows us to make the degree effective. 

\textbf{Acknowledgements}
\quad The author would like to thank Will Sawin for suggesting ideas that improved the strength of the bounds in multiple parts of this paper.  The author was partially supported by National Science Foundation Grant Number DGE-2036197.

\section{Background} 
\label{sec: background}
The purpose of this section is to explain the idea behind the proof of \cite[Theorem 10.1]{LV} and why it applies to the moduli space of hypersurfaces in $\P^n$ with fixed degree, modulo checking the numerical conditions involving the Hodge numbers.  For a thorough treatment with complete proofs, we refer the reader to \cite[Sections 9, 10, 11]{LV}. 

\subsection{Setup}

We begin with a smooth proper morphism $\pi\colon X\rightarrow Y$ over $\Z[S^{-1}]$ whose fibers are geometrically connected of relative dimension $n-1$\footnote{We choose the variable $n$ because the main example of interest will be the moduli space of hypersurfaces of degree $d$ in $\P^n$.}.   We wish to bound the set of integral points $Y(\Z[S^{-1}])$.  The morphism $\pi$ gives rise to a Hodge-Deligne system $R^{n-1}\pi_*\Q$ in the sense of \cite[Section 5]{LV}, which loosely consists of the data of four cohomology theories: a singular local system, an \'etale local system, a vector bundle, and a filtered isocrystal on $X$ that satisfy various compatibility isomorphisms.  

The stalk of the singular local system is given by the singular cohomology of the fibers $H^{n-1}(X_y, \Q)$.  The primitive part $V_y = H^{n-1}(X_y, \C)^{\prim}$ carries an intersection form and a polarized Hodge structure.  The monodromy representation (referred to in the statement of Theorem \ref{thm: 10.1}) is the natural mapping from the fundamental group of the base into the automorphism group
\[
\pi_1(Y) \rightarrow \Aut(H^{n-1}(X_y, \C)^{\prim}).
\]
The image is considered `large' if its Zariski closure contains the special orthogonal or symplectic group, whichever one preserves the intersection form. 

The Lie algebra of the generalized automorphism group of $V_y$ is given by either $\C\oplus \Sym^2V_y$ or $\C\oplus \wedge^2 V_y$, depending on the intersection form.  The Hodge structure on $V_y$ induces a weight 0 Hodge structure on this Lie algebra, known as the adjoint Hodge structure  Thus the Hodge numbers  of $V_y$ determine the \textit{adjoint Hodge numbers} $h^p$, the dimension of the $(p, -p)$ component in the adjoint Hodge structure.  This leads to the formula \cite[p. 57]{LV} 
\begin{equation}
\label{eq: hp}
2h^p = \sum_{p_1+p_2 = p+n-1} h^{p_1, n-1-p_1}h^{p_2, n-1-p_2}\pm h^{(p+n-1)/2, (-p+n-1)/2}   
\end{equation}

where we take the positive sign for if the intersection form is symplectic and the negative sign if it is orthogonal.  

\subsection{$p$-adic period maps and transcendence results}

Knowledge of the monodromy representation and the distribution of the adjoint Hodge numbers is used in the following way. 

The fiber at an integral point $y\in Y(\Z[S^{-1}])$ of the \'etale local system $R^{n-1}\pi_*\Q_l$ gives a Galois representation of the absolute Galois group 
\[
\rho_y\colon \Gal(\Q)\rightarrow H^{n-1}(X_y, \Q_l). 
\]
By a result of Faltings, there are only finitely many possibilities for the semisimplification of this representation.  The big monodromy results and numerical conditions will ensure that not too many distinct points give the same representation and deal with the issue of semisimplification. 

To analyze the variation of this representation with $y$, we pick a place $v$ and pass to the corresponding $p$-adic representations, which are crystalline.  
Then $p$-adic Hodge theory associates to these crystalline representations the data of a vector space with a semilinear Frobenius action and a Hodge filtration.  
The Gauss-Manin connection identifies the vector space and Frobenius of nearby fibers, but the $p$-adic period map determines how the Hodge filtration varies. 
 
Two points correspond to the same representation when an element of the centralizer of the Frobenius in the monodromy group transforms one Hodge filtration into the other.  By bounding the size of the Frobenius centralizer and enforcing appropriate numerical conditions, we can ensure this does not occur too often.

More precisely, given a $p$-adic disk $U_p$ around $y$ in $Y(\Q_p)$, the $p$-adic period map
\[
\Phi_p\colon U_p \rightarrow \mf{h}_{\Q_p}\subset \mf{h}^*_{\Q_p}
\]
is a $p$-adic analytic map from $U_p$ into a polarized period domain $\mf{h}_{\Q_p}$.  This can be viewed as a subset of the flag variety $\mf{h}^*_{\Q_p}$ parameterizing isotropic flags, those flags satisfying an appropriate orthogonality condition on its filtered components.  The $p$-adic version of the Bakker-Tsimerman transcendence theorem implies that the inverse image of a subvariety of $\mf{h}^*_{\Q_p}$ with codimension greater than or equal to that of $Y$ is not Zariski dense in $Y$. 

In our scenario, the flag variety is $\sum_{p>0}h^p$ and the dimension of the Frobenius centralizer can be bounded above by $h^0$ \cite[Lemma 10.4]{LV}.  This explains the first numerical condition in Theorem \ref{thm: 10.1}.  The second condition is used to deal with the issue of semisimplification, and is significantly more technical and restrictive.  It is obtained through carefully tracing out what it means for two representations to have the same semisimplification in terms of their period maps; for details see \cite[Lemma 10.5]{LV}.  The resulting condition is stated in terms of a function $T$, where $T(E)$ is loosely defined to be the sum of the sum of the $E$ topmost adjoint Hodge numbers.  For instance, if $p_{\max}$ is the largest $p$ with $h^p\neq 0$, then $T(1)=p_{\max}$.  More generally, 
\[
T\colon [0, \sum_{j\in \Z}h^j] \rightarrow \R_{\ge 0}
\]
is the unique piecewise-linear function satisfying $T(0)=0$ and 
\[
T'(x) = \begin{cases}
    p_{\max} & x\in (0, h^{p_{\max}}) \\ 
    p_{\max}-1 & x\in (h^{p_{\max}}, h^{p_{\max}} + h^{p_{\max}-1}) \\ 
    \cdots.
\end{cases}
\]
In our scenario, we let $X\rightarrow Y$ be the universal smooth hypersurface of degree $d$.  To obtain the desired Zariski non-density statement we need to compute these adjoint Hodge numbers $h^p$ from the Hodge numbers of the hypersurfaces and check they satisfy the numerical conditions listed in Theorem \ref{thm: 10.1}, and also check that the big monodromy statement is fulfilled.  The latter was already achieved in \cite{Beauville}.

\subsection{The monodromy group of the moduli space of hypersurfaces of fixed degree}
The monodromy group of the moduli space of smooth hypersurfaces in $\P^n$ of degree $d$ was calculated in \cite{Beauville}.  We state these results below.  

    Let $U_{n-1, d}$ be the moduli space of smooth hypersurfaces of degree $d$ in $\P^{n}$. 
 Let $u$ be a point in $U_{n-1, d}$ and let $X$ be the corresponding hypersurface.  Let $\Gamma_{n-1, d}$ be the image of the monodromy representation 
 \[
\rho\colon \pi_1(U_{n-1, d},u)\rightarrow \Aut(H^{n-1}(X, \Z)).
 \]

If $n$ is odd, let $h\in H^2(X, \Z)$ be the class of a hyperplane section and let $\Ort_h(H^{n-1}(X, \Z))\subset \Ort(H^{n-1}(X, \Z))$ be the subgroup preserving $h^{(n-1)/2}$.  Similarly, given a quadratic form $q$ let $\Spo(L, q)\subset \Sp(L)$ be the subgroup preserving $q$.
 
\begin{theorem}
\cite[Theorem 2]{Beauville} 
    If $n$ is odd and $d>4$, then $\Gamma_{n, d}=\Ort_h(H^n(X, \Z))$. 

    \cite[Theorem 4]{Beauville} 
    If $n$ is even and $d$ is even, then $\Gamma_{n, d}=\Sp(H^n(X, \Z))$.  If $n$ is even and $d$ is odd, then there exists on $H^n(X, \Z)$ a quadratic form$\pmod 2$ $q_X$ invariant by monodromy, and we have $\Gamma_{n, d}=\Spo(H^n(X, \Z), q_X).$
\end{theorem}

Therefore, the monodromy representation has large image when $d>4$, and we are reduced to checking the numerical conditions in Theorem \ref{thm: 10.1}.

\section{Effective bounds on the dimension}
    \subsection{Setup and strategy} 
We wish to check find $n$ for which the numerical conditions in Theorem \ref{thm: 10.1} hold for sufficiently large $d$.
We will use the same notation as \cite[Section 10]{LV} throughout this appendix, which we review here.  Let $X_n$ be the random variable indicating $\frac{n-1}{2}$ less than the number of descents to a uniformly distributed random permutation of $\{1, 2, \ldots, n\}$.  This is precisely the normalized distribution of the Eulerian numbers.  By \cite{Stanley}, $P(X_n =p - \frac{n-1}{2})$ is the probability that a sum of $n$ uniform, independent and identically distributed variables between 0 and 1 is between $p$ and $p+1$.  

Next, let $X'_n$ be the random variable given by a sum of two distinct copies of $X_n$.  Let $\beta_p = P(X'_n = p)$. 

The relevance of these  distributions is that the Hodge numbers of a degree $d$ hypersurface in $\P^n$ satisfy 
\[
h^{p, q}\sim \frac{d^n}{n!}A(n, p)
\]
where the Eulerian number $A(n, p)$ is the number of permutations of $[n]$ with $p$ descents.  We will need to make this approximation precise to make $d$ effective given $n$.  However, to make $n$ effective it suffices to get better bounds on these distributions.  Therefore we will use this statement now and give a proof of a stronger statement in Section \ref{sec: deg}. 

\subsection{Explicit bounds on $\beta_p$}
We will need to have an upper bound on $\beta_0$ and a lower bound on $\sum_{p>0}\beta_p$.

\begin{lemma}
\label{lem: upp}
    For all $n$, we have $\beta_0 \le \frac{\sqrt{3}}{\sqrt{n+4}}$. 
\end{lemma}
\begin{proof}
We will use a similar strategy to that employed in Lemma 10.4 of \cite{LS}. 

We use the fact that the sequence $\beta_i$ is log-concave and satisfies $\beta_i = \beta_{-i}$.  This is true because the Eulerian numbers are log-concave and log-concavity is preserved under convolution. 
Furthermore, the second moment of $\beta_i$ is $\sum_i i^2\beta_i = \frac{n+1}{6}$ by \cite[Eqn. 10.10]{LV}.  

First, we claim that for all $k\ge 1$, we have $\sum_{i = k}^{\infty} \beta_i \le \sum_{i = k}^{\infty}\beta_0t^i$, where $\frac{2\beta_0t}{1-t} = 1-\beta_0$.  Indeed, with this choice of $t$ we have $\sum_{i > 0} \beta_0 t^i = \frac{1-\beta_0}{2} = \sum_{i>0}\beta_i$.  Thus if the claim is false for some $k$, we must have $\beta_j < \beta_0t^j$ for some $j<k$.  Then by log-concavity of the $\beta_i$ we have $\beta_i < \beta_0t^i$ for all $i > j$ and in particular for $i =k$, contradiction.  

    Therefore, we have 
    \[
\frac{n+1}{12} = \sum_{i=1}^{\infty}i^2\beta_i \le \sum_{i=1}^{\infty} i^2\beta_0t^i = \frac{t(1+t)\beta_0}{(1-t)^3} = \frac{1-\beta_0}{2}\cdot \frac{1+\beta_0}{2\beta_0}\cdot \frac{1}{\beta_0} = \frac{1-\beta_0^2}{4\beta_0^2}.
    \]
This implies $\beta_0\le \sqrt{\frac{3}{n+4}}$ as desired. 
\end{proof}

By the central limit theorem, we know that $X'(n)/\sqrt{n}$ converges to the normal distribution $N(0, \frac{1}{6})$ as $n\rightarrow\infty$.  Note that $\sum_{p>0}p\beta_p$ is half the expected absolute deviation from the mean of $X'(n)$.  The deviation from the mean of a normal distribution is given by a half-normal distribution.  Therefore, 

\[
\sum_{p>0}p\beta_p \sim \frac{1}{2} \cdot \sqrt{\frac{2}{\pi}} \cdot \sqrt{\frac{n}{6}} = \sqrt{\frac{n}{12\pi}}
\]
as $n\rightarrow \infty$.
However, to obtain a strict lower bound for $\sum_{p>0}p\beta_p$ we will need to employ additional arguments that give a strict lower bound for $E[|X'(n)|]$.  Even though the bound we use will not be optimal, its linearity will be enough for us to conclude.  

Let us first compute the expected deviation from the mean of the Irwin-Hall distribution.  Let $IH(n)$ be the distribution given by the sum of $n$ uniform independently and identically distributed variables between $[0, 1]$.    
\begin{lemma}
\label{lemma: IH}
    $E[|IH(n) - \frac{n}{2}|] \ge \frac{\sqrt{n}}{2\sqrt{6}}$. 
\end{lemma}
\begin{proof}
    Let $Z(n)$ be the distribution given by $IH(n) - \frac{n}{2}$ and let $Y(n) = |Z(n)|$.  Then by H\"older's inequality, we have $E[Y(n)]^{2/3}E[Y(n)^4]^{1/3}\ge E[Y(n)^2]$.  We have $E[Y(n)^2] = E[Z(n)^2] = \frac{n}{12}$.  We have 
    \[
E[Y(n)^4] = nE[Z(n)]^4 + 3n(n-1)E[Z(n)^2] = \frac{n}{80} + \frac{n(n-1)}{72} \le \frac{n^2}{72}.  
    \]

    Thus we have 
    \[
E[Y(n)] \ge \left(\Big(\frac{n}{12}\Big)\Big(\frac{n^2}{72}\Big)^{-1/3}\right)^{3/2} = \frac{\sqrt{n}}{2\sqrt{6}}, 
    \]
    as desired. 
\end{proof}

\begin{remark}
    Through numerical experiments, it appears very likely that the limiting value of $\sqrt{\frac{n}{6\pi}}$ is in fact a lower bound.
\end{remark}

\begin{lemma}
\label{lemma: sym}
    We have $E[|X'(n)|] \ge E(|A-B|)-1$, where $A, B$ are chosen independently from $IH(n)$.  
\end{lemma}
\begin{proof}
    By \cite{Stanley}, we can compute the distribution of $X'(n)$ by taking two independent variables $A, B$ from $IH(n)$ and setting $X'(n)$ to be $\lfloor A\rfloor + \lfloor B\rfloor -(n-1)$.  Note that by symmetry, we can replace $B$ with $n-B$.  Doing this, we see that the corresponding value of $X'(n)$ differs from $A-B$ by at most 1.  The desired result follows.
\end{proof}

\begin{lemma}
\label{lem: worse low}
    We have $\sum_{p>0}p\beta_p \ge \frac{\sqrt{n}}{4\sqrt{3}}-\frac{1}{2}$ for all $n>0$.    
\end{lemma}
\begin{proof}
Recall that $\sum_{p>0}p\beta_p=\frac{1}{2}E[|X'(n)|]$, which by Lemma \ref{lemma: sym} is at least $\frac{1}{2}E(|A-B|)-\frac{1}{2}$ where $A, B$ are chosen independently from $IH(n)$.  By symmetry, note that $E(|A-B|) = E[|A - (n-B)|] = E[|IH(2n) -n] \ge \frac{\sqrt{2n}}{2\sqrt{6}}$.  The desired result follows.
\end{proof}

\subsection{Explicit bounds}

We can now verify the numerical conditions.  We first make the dimension $n$ effective.  To do this, we will use the fact that for fixed $n$, as $d\rightarrow\infty$ we have 
\[
h^{pq}(d) \sim \frac{d^n}{n!}A(n, p). 
\]
As stated before, we will prove a more precise version of this result in the following section. 

\begin{proposition}
\label{prop: n eff}
For all $n > 2000$, there exists some degree $d(n)$ such that the hypersurfaces of degree $d(n)$ in $\P^n$ with good reduction outside $S$ are not Zariski dense in their moduli space.  
\end{proposition}
     
Like in \cite[Proposition 10.2]{LV}, we first simplify the conditions we wish to check. 

\begin{lemma}
\label{lemma: 1.26}
    It suffices to show $\sum_{p>0}ph^p > 2T(1.26h^0)$. 
\end{lemma}

\begin{proof} 
The moduli space $Y$ of degree $d$ hypersurfaces in $\P^n$ has dimension 
\[
\binom{n+d}{d-1}-1 \sim \frac{d^{n+1}}{(n+1)!}. 
\]
Thus we have 
\[
\lim_{d\rightarrow\infty} \frac{\dim(Y)}{h^0} = 0. 
\]
Then the first condition 
\[
\sum_{p>0}h^p \ge h^0 +\dim(Y) 
\]
is clearly satisfied for sufficiently large $d$, in light of Lemma \ref{lem: upp}. 
The second condition states 
\[
\sum_{p>0}ph^p > T(h^0 + \dim (Y)) + T(\frac{3}{2}h^0 + \dim (Y)). 
\]
We recall that the slope of the graph of $T$ is non-increasing.  Thus as $d\rightarrow\infty$, we have 
\[
T(1.26h^0) - T(h^0+\dim(Y)) > T(\frac{3}{2}h^0 + \dim(Y)) - T(1.26h^0), 
\]
so the result follows.

\end{proof} 

\begin{proof}
[Proof of Proposition \ref{prop: n eff}] 
  Let $H = \sum_{p} h^p$ and $H_1 = \sum_{p>0} ph^p$. 
  We wish to show that $H_1 > 2T(1.26 h^0)$. 

  We know that $\beta_p\sim \frac{h^p}{H}$.  Thus by Lemma \ref{lem: worse low}, we have 
  \[
  H_1 = \sum_{p>0}Hp\beta_p + p(h^p-H\beta_p) \ge \frac{H\sqrt{n}}{4\sqrt{3}} - \frac{H}{2} + \sum_{p>0}p(h^p-H\beta_p).
  \]

  On the other hand, by Lemma \ref{lem: upp}, we have 
  \[
2T(1.26 h^0) = 2T(1.26 H\beta_0 + 1.26 (h^0 - H\beta_0)) < 2T\left(1.26 H\sqrt{\frac{3}{n+4}} + 1.26 (h^0 - H\beta_0)\right). 
  \]

  For sufficiently large $d$ we have $|h^p-H\beta_p| < \eps H$ for all $p$ and $\eps > 0$.  
  
  Thus, it suffices to show that 
    \[
 T\left(\frac{1.26 H\sqrt{3}}{\sqrt{n}}\right) <  \frac{H\sqrt{n}}{4\sqrt{3}} - \frac{H}{2}.
    \] 
    
    Let $\eps = \frac{1}{24\cdot 1.26}$.  We can separate the LHS into the contribution of Hodge numbers below $\eps n$ and those above $\eps n$.  The contribution of the Hodge numbers below $\eps n$ is at most 
    \[
    T\left(\frac{1.26 H\sqrt{3}}{\sqrt{n}}\right) - T(\eps n) \le (\eps n)\frac{1.26H\sqrt{3}}{\sqrt{n}} = 1.26H\eps\sqrt{3n}. 
    \]
    For the Hodge numbers above $\eps n$, we can use the variance bound $\sum_{p>0} p^2\beta_p = \frac{n+1}{12}$.  Then 
    \[
    T(\eps n) = \sum_{p > \eps n}pH\beta_p + \sum_{p>\eps n}(h^p - H\beta_p) < H(\eps n)^{-1}\frac{n+1}{12} + \sum_{p>\eps n}(h^p - H\beta_p) =  \frac{(n+1)H}{12 n \eps} + \sum_{p>\eps n}(h^p - H\beta_p).
    \] 
    For any $\eps' >0$, the last summand is at most $H\eps'$ for sufficiently large $d$.

    Adding these together, we see that it suffices to prove that 
\[
1.26H\eps\sqrt{3n} + \frac{(n+1)H}{12 n \eps} + H\eps' + \frac{H}{2} \le \frac{H\sqrt{n}}{4\sqrt{3}}.  
\]

By our choice of $\eps$, the first term is half of the RHS.  Thus we see that the inequality is satisfied when 
\[
\frac{2.52(n+1)}{n} + \frac{1}{2} \le \frac{\sqrt{n}}{8\sqrt{3}},
\]
$\sqrt{n}> 8\sqrt{6}\cdot 2.52\sqrt{2}\cdot \frac{n+1}{n}$, which holds when $n> 2000$. 
\end{proof}

\section{Effective bounds on both the dimension and degree}
In this section we take $d\ge n> 500000$.  We will work with the Hodge numbers directly rather than examining their asymptotic behavior as the degree goes to infinity, which will allow us to make the value of the degree effective. 

\begin{remark}
    As the bounding arguments are technical and may seem rather arbitrary, we give a brief explanation of how the bounds for $n$ and $d$ are used.  In Lemma \ref{lemma: var} we give a certain variance bound which does not require $n$ and $d$ to be particularly large, but in Lemmas \ref{lemma: var upper} and \ref{lem: upp 2} we apply it in a simplified form where the simplification requires $n$ and $d$ to be large.  We need $n$ and $d$ to be large for Corollary \ref{cor: z'd} and Lemma \ref{lemma: bounds} for similar reasons.  Next, to check the numerical conditions there is a term involving the dimension of the moduli space itself which is generally supposed to be insignificant, but if $n$ is much larger than $d$ then it can become a problem; this is why we set $d\ge n$.  Finally, in Proposition \ref{prop: n d effective} the final computation is what forces $n$ to be rather large.  With the same methods one could prove weaker versions of the previous lemmas that would not require $n$ and $d$ to be so large, but this would  weaken the final argument which would force $n$ to be larger than otherwise.  Nevertheless, it is certainly expected that closer analysis would improve the bounds. 
 Once $n$ and $d$ reach a certain point (e.g. the bound of 2000 used in the previous section), the distribution of the Hodge numbers is surely close enough to other known distributions (e.g. the one used in the previous section) so that the desired bounds hold.  However, in order to give proofs one has to have a handle on their exact values, which is why in this section we use some looser bounds that are easier to work with. 
\end{remark}

\label{sec: deg}
\subsection{Hodge numbers of hypersurfaces} 
The computation of the Hodge numbers of a smooth degree $d$ hypersurface $X\subset \P^n$ is classical.  The formula we will use for them is stated below.  
\begin{proposition}
    \cite[Corollary 17.5.4]{Arapura}
      Let $h^{p, n-1-p}(d)$ be the $p$th Hodge number of a smooth degree $d$ hypersurface $X\subset \P^n$.  Then $h^{p, n-1-p}(d)-\delta_{p, n-1-p}$ is the coefficient of $t^{(p+1)d}$ in $(t+t^2+\cdots + t^{d-1})^{n+1}$. 
\end{proposition} 

We will work with this formula directly rather than appealing to the asymptotic behavior of the Hodge numbers so that we can obtain effective statements on the degree.  As in the previous section, we will define and compute statistics of distributions created from the Hodge numbers.

  Let $Y_d(n)$ denote the distribution of the sum of $n$ numbers randomly drawn from the set $\{\frac{1}{d}, \frac{2}{d} \ldots, \frac{d-1}{d}\}$.  Let $Z_d(n)$ denote the same distribution, except conditioning on the event that the sum is an integer.  Ultimately we need to deal with the adjoint Hodge numbers, which are given by a convolution of $Z_d(n)$ with itself (with a possible difference of 1 in $h^{\frac{n-1}{2}, \frac{n-1}{2}}$), and then adding or subtracting a smaller term.  In view of this, let $Z'_d(n)$ be the sum of two distinct copies of $Z_d(n)$.  Let $\gamma_p = P(Z_d'(n) = p)$.  Note that the dependence of $\gamma_p$ on the degree $d$ is implicit. 
  
  We will now bound statistical properties of $Z'_d(n)$ and then deal with the minor complications relating them to those of the adjoint Hodge numbers.
  
\subsection{Explicit bounds on $\gamma_p$} 
As in the previous section, we will need to have an upper bound on $\gamma_0$ and a lower bound on $\sum_{p>0}p\gamma_p$.  To achieve these, we will need some bounds on the statistics of $Y_d(n)$ and $Z_d(n)$.  

Let $\alpha_q$ be the probability of $q$ occurring in $Y_d(n)$ and let $\beta_m$ be the probability of $m$ occurring in $Z_d(n)$.  Note that $m$ is an integer while $q$ is just an integer multiple of $\frac{1}{d}$. 

\begin{lemma}
\label{lemma: recur}
    We have 
    \[
\alpha_m = \left(\frac{1}{d} + \frac{(-1)^n}{d(d-1)^{n-1}}\right) \beta_m. 
    \]
\end{lemma}
\begin{proof}
    By definition, $\alpha_m = P(d|Y_d(n))\beta_m$.  Note that we have 
    \[
P(d|Y_d(n)) = \frac{1-P(d|Y_d(n-1))}{d-1}. 
    \]
    Letting $q_n = P(d|Y_d(n))-\frac{1}{d}$, we obtain the recurrence relation 
    \[
q_n = -\frac{q_{n-1}}{d-1}.
    \] 
    Then because $q_1 = -\frac{1}{d}$, the result follows. 
\end{proof} 
For convenience, let $c_d = P(d|Y_d(n))$ (with the dependence on $n$ implicit); we have shown that $c_d$ is numerically close to $\frac{1}{d}$.  

The next lemma will allow us to give an upper bound on $\gamma_0$.
\begin{lemma}
\label{lemma: var}
    The variance of $Z'_d(n)$ satisfies $\Var(Z'_d(n)) \ge  \frac{n(d-3)}{6d} - \sqrt{3n}$. 
\end{lemma}
\begin{proof}
    We have $\Var(Z'_d(n)) = 2\Var(Z_d(n))$.  To compute $\Var(Z_d(n))$, we first compute $\Var(Y_d(n))$.  By direct computation, it is given by 
    \[
\Var(Y_d(n)) = \frac{2n}{d^2(d-1)}\sum_{i=1}^{\lfloor(d-1)/2\rfloor}i^2 = \begin{cases}
    \frac{n(d+1)}{12d} & d \text{ odd} \\ 
    \frac{n(d-2)}{12d} & d \text{ even}.
\end{cases}
    \]

We have that $\Var(Z_d(n)) = 2\sum_{p>0}{p^2}\beta_p$.  In order to give a lower bound on this, we compare this to the known value of  
\[
\Var(Y_d(n)) = 2\sum_{q>0} q^2\alpha_q. 
\]
Note that here $q$ ranges over integer multiples of $\frac{1}{d}$.  We have 
\[
2dc_d\sum_{p>0}(p+1)^2 \beta_p + 1 > 2\sum_{q>0}q^2\alpha_q. 
\]
This is because multiplying by $c_d$ turns $\beta_p$ into $\alpha_p$, multiplying by $d$ and using $(p+1)^2$ instead of $p^2$ accounts for the contribution to the variance of $q\in [p, p+1)$, and the 1 accounts for the contribution of $q<1$. 

Now we have 
\[
\sum_{p>0}(p+1)^2\beta_p = \sum_{p>0}p^2\beta_p+ \sum_{p>0}(2p+1)\gamma_p,
\] 
and furthermore by Cauchy-Schwarz we have  $2\sum_{p>0}p\beta_p\le \sqrt{2\sum_{p>0}p^2\beta_p}$. 

For convenience, letting $V_d = \Var(Z_d(n)) = 2\sum_{p>0}p^2\beta_p$, we have 
\[
dc_d(V_d + 2\sqrt{2}\sqrt{V_d}+ 1) + 1 > \Var(Y_d(n)). 
\]
We can then give a bound of $V_d\ge \frac{n(d-3)}{12d} - \frac{\sqrt{3}}{2}\sqrt{n}$, so $\Var(Z'_d(n)) \ge  \frac{n(d-3)}{6d} - \sqrt{3n}$ as desired. 
\end{proof}

Later in the final argument, we will also use an upper bound for this variance.  The following upper bound can be proven in the same way, by analyzing $\sum_{p>0}(p-1)^2\beta_p$ rather than $\sum_{p>0}(p+1)^2\beta_p$.   We state it separately so that it is clear where each statement is used. 

\begin{lemma}
\label{lemma: var upper}
    We have $\Var(Z'_d(n)) \le  \frac{n}{5.1}$.
\end{lemma}

We can now give an upper bound on $\gamma_0$. 

\begin{lemma} 
\label{lem: upp 2}
We have $\gamma_0 \le \frac{\sqrt{12.8}}{2\sqrt{n}}$.
\end{lemma}
\begin{proof}
    We note that since the convolution of log-concave distributions is log-concave, then $Z_d(n)$ is log-concave.  Taking every $d$th element of a log-concave sequence preserves log-concavity.  Thus by the argument in the proof of Lemma \ref{lem: upp} and the variance calculation in Lemma \ref{lemma: var}, we have 
    \[
\frac{n(d-3)}{12d} - \sqrt{3n} \le \frac{1-\gamma_0^2}{4\gamma_0^2}. 
    \]
    Because we are taking $d\ge n\ge 500000$, the left hand side is at least $\frac{n}{12.8}$.  Making this simplification, we obtain $\gamma_0\le \frac{\sqrt{12.8}}{2\sqrt{n}}$ as desired. 
\end{proof}

Next, computing $\sum_{p>0}p\gamma_p$ is tantamount to computing the expected deviation of $Z'_d(n)$ from the mean.  We first do this for $Y_d(n)$ and $Z_d(n)$.  By definition, we have 
\[
E[|Y_d(n) - \frac{n}{2}|] = 2 \sum_{q>\frac{n}{2}} \left(q-\frac{n}{2}\right)\alpha_q.  
\]
while 
\[
E[|Z_d(n) - \frac{n}{2}|] = 2 \sum_{m>\frac{n}{2}} \left(m-\frac{n}{2}\right)\beta_m,  
\]

We will compare these using the fact that both distributions are increasing until $\frac{n}{2}$ and decreasing afterwards. 

  \begin{lemma}
  \label{lemma: Yd}
      We have 
      \[
      E[|Y_d(n) - \frac{n}{2}|] \ge  E[|IH(n) - \frac{n}{2}|] - \frac{n}{2d}.
      \] 
  \end{lemma}
  \begin{proof}
  Note that the distribution $Y_d(n)$ can be obtained from $IH(n)$ in the following way: for each of the $n$ numbers $\alpha\in (0, 1)$, replace $\alpha$ with $\frac{\lfloor(d-1)\alpha\rfloor+1}{d}$. 
 This changes $\alpha$ by no more than $\frac{1}{d}$, so we have 
\[
E[|Y_d(n)-\frac{n}{2}|] \ge E[|IH(n) - \frac{n}{2}|] - \frac{n}{d}
\]
as desired.
  \end{proof}

\begin{lemma}
\label{lemma: Zd}
    We have 
    \[
E[|Z_d(n) - \frac{n}{2}|] \ge \left(1-\frac{1}{(d-1)^{n-1}}\right)E[|Y_d(n) - \frac{n}{2}|] - 1.
    \]
\end{lemma}
\begin{proof}
Let $\parity(n)\in \{0, 1\}$ denote the parity of $n$. First we will bound the contribution to $E[|Y_d(n) - \frac{n}{2}|]$ coming from when $|Y_d(n) - \frac{n}{2}| < 1-\frac{\parity(n)}{2}$.  Even with the extremely crude bound obtained by letting the probability of the above event be bounded above by 1, we obtain that this contribution to the expected deviation from the mean is no more than 1.  

Next, we will compare the contribution to $E[|Z_d(n) - \frac{n}{2}|]$ coming from an integer $m>\frac{n}{2}$ to the contribution to $E[|Y_d(n) - \frac{n}{2}|]$ coming from $m+\frac{i}{d}$ for $0\le i\le d-1$.  For convenience, write $p =m-\frac{n}{2}$.  Because $\alpha_m \ge \alpha_{m+\frac{i}{d}}$, we have
    \[
\sum_{i=0}^{d-1} (p+\frac{i}{d})\alpha_{m+\frac{i}{d}} \le pd\alpha_{m} + \frac{d-1}{2}\alpha_m
    \]
    Next, by Lemma \ref{lemma: recur} we have 
    \[
\beta_m \ge \left(1-\frac{1}{(d-1)^{n-1}}\right)d\alpha_m. 
    \] 
    
    Summing over all $m$, we obtain 
    \[
E[|Y_d(n) - \frac{n}{2}|] \le 2\left(1-\frac{1}{(d-1)^{n-1}}\right)^{-1}\sum_{m>\frac{n}{2}} \left(\left(m - \frac{n}{2}\right)\beta_m + \frac{d-1}{2}\alpha_m\right). 
    \]

    In the given range of $n$ and $d$, the $\frac{d-1}{2}\alpha_m$ is less than $\frac{\beta_m}{2}$.  Then since $\sum_m \beta_m =1$, we obtain 
    \[
E[|Z_d(n) - \frac{n}{2}|] \ge \left(1-\frac{1}{(d-1)^{n-1}}\right)E[|Y_d(n) - \frac{n}{2}|] - 1,
    \]
    as desired. 
\end{proof}

Finally, to obtain a lower bound on $E[|Z'_d(n) - n|]$ we will simply use the crude bound $E[|Z'_d(n) - n|] > E[|Z_d(n) - \frac{n}{2}]$.  Even though one would expect it to be $\sqrt{2}$ times this in the limit, having a bound weaker by a constant factor will suffice for our purposes.

\begin{corollary}
\label{cor: z'd}
    We have $E[|Z'_d(n) - n|] >  \frac{\sqrt{n}}{5.2}$.
\end{corollary}
\begin{proof}
    By Lemmas \ref{lemma: Yd}, \ref{lemma: Zd}, and \ref{lemma: IH}, we have 
    \begin{align*}
E[|Z'_d(n) - n|] &> \left(1-\frac{1}{(d-1)^{n-1}}\right)E[|IH(n) - \frac{n}{2}|] - \frac{3}{2} \\ 
& \ge \left(1-\frac{1}{(d-1)^{n-1}}\right)\frac{\sqrt{n}}{2\sqrt{6}} - \frac{3}{2} \\ 
& \ge \frac{\sqrt{n}}{5.2},
    \end{align*}
as desired. 
\end{proof}

\subsection{Estimate of distribution of the adjoint Hodge numbers}
We now need to compare the distribution $\beta_p$ with the actual distribution of the adjoint Hodge numbers.  We recall the formula \cite[p. 57]{LV}
\begin{equation}
\label{eq: hp2}
2h^p = \sum_{p_1+p_2 = p+n-1} h^{p_1, n-1-p_1}h^{p_2, n-1-p_2}\pm h^{(p+n-1)/2, (-p+n-1)/2}
\end{equation}
with the sign depending on the intersection form.  In both cases, we want to compare these distributions to the one obtained through convolution of the Hodge numbers without the extra minor term.  That is, to the distribution defined by  
\begin{equation}
\label{eq: hp'}
h'^p = \sum_{p_1+p_2 = p+n-1} h^{p_1, n-1-p_1}h^{p_2, n-1-p_2}. 
\end{equation}
Let the sum of all the adjoint Hodge numbers in these cases be $H$ and $H'$ respectively, so that the distributions are given by $\frac{h^p}{H}$ and $\frac{h'^{p}}{H'}$. 

The difference between these distributions is extremely minor in the range of $n$ and $d$ we choose.  Indeed, note that each $h^p$, after dividing by 2, differs from $h'^p$ by less than the sum of all the Hodge numbers, while the sum of all the $h'^p$ is given by the square of the sum of all the Hodge numbers.  The sum of all the Hodge numbers is given by $c_d(d-1)^{n+1}$ (or 1 more in the case of odd $n$).  Thus for each $p$, the magnitude of the difference between $\frac{h^p}{H}$ and $\frac{h'^{p}}{H'}$ can be bounded by something on the order of $\frac{1}{d^n}$.  With this in mind, we can give slightly weaker bounds. 

Let $H_1 = \sum_{p>0} ph^p$.  We want an upper bound for $\frac{h^0}{H}$ and a lower bound for $\frac{H_1}{H}$.  By Lemma \ref{lem: upp 2} and Corollary \ref{cor: z'd} and the discussion above, we have the following bounds. 

\begin{lemma}
\label{lemma: bounds}
For $d\ge n\ge 40000$, we have 
    \[
h^0 < \frac{\sqrt{13}H}{2\sqrt{n}} 
    \]
    and
    \[
H_1 > \frac{H\sqrt{n}}{11}.
    \]
\end{lemma}

\subsection{Effectivity for both dimension and degree} 
We can now run the argument in the previous section again to obtain an effective statement on Zariski non-density.  The first place this occurs is Lemma \ref{lemma: 1.26}, which reduces the numerical conditions to showing that 
\[
\sum_{p>0}ph^p > 2T(1.26 h^0). 
\]
At some additional cost to the sharpness of our results, we will instead show that 
\begin{equation}
\label{eq: weak ineq}
    \sum_{p>0}ph^p > 2T(2h^0).
\end{equation}
for our choices of $n$ and $d$.  Indeed, recall that in order to apply Theorem \ref{thm: 10.1}, the numerical conditions to be checked are 
\[
\sum_{p>0} h^p\ge h^0+\dim(Y), \quad \sum_{p>0} ph^p\ge T(h^0+\dim(Y)) + T\left( \frac{3}{2}h^0 + \dim(Y)\right). 
\]
Then if $\dim(Y)\le \frac{h^0}{2}$, then showing inequality \ref{eq: weak ineq} is clearly sufficient.  But we have $\dim(Y)=\binom{n+d}{d-1}-1$.  On the other hand, we have 
\[
h^0 \ge \frac{1}{2n}\left(\sum_{p}h^{p,n-1-p}\right)^2 \pm  \eps
\] 
by Cauchy-Schwarz, where the $\eps$ accounts for the extra minor terms dealt with in the previous subsection.  In any case, since $\sum_p h^{p, n-1-p}$ is given by $c_d(d-1)^{n+1}$ (or 1 more in the case of odd $n$) we see that $h^0 \gg 2\dim(Y)$ for the values of $n$ and $d$ we are looking at.  Thus it remains to check that $\sum_{p>0}ph^p > 2T(2h^0)$. 

\begin{proposition}
\label{prop: n d effective}
    For all $d\ge n>500000$, the hypersurfaces of
degree $d$ in $\P^n$ with good reduction outside $S$ are not Zariski dense in their moduli space. 
\end{proposition}
\begin{proof}
  We wish to show that $H_1 > 2T(2 h^0)$, and by Lemma \ref{lemma: bounds} we have $h^0 < \frac{\sqrt{13}H}{2\sqrt{n}} $
    and $H_1 > \frac{H\sqrt{n}}{11}$.
  Thus it suffices to show that 
  \[
T\left(\frac{\sqrt{13}H}{\sqrt{n}}\right) < \frac{H\sqrt{n}}{22}. 
  \]
  
   Let $\eps = \frac{1}{44\sqrt{13}}$.  We can separate the LHS into the contribution of Hodge numbers below $\eps n$ and those above $\eps n$.  The contribution of the Hodge numbers below $\eps n$ is at most 
    \[
    T\left(\frac{\sqrt{13}H}{\sqrt{n}}\right) - T(\eps n) \le (\eps n)\frac{\sqrt{13}H}{\sqrt{n}} = \sqrt{13}H\eps\sqrt{n}. 
    \]
    For the Hodge numbers above $\eps n$, we can use a bound on $\sum_{p>0} p^2\frac{h^p}{H}$.  Recall that In Lemma \ref{lemma: var upper}, we showed that $\Var(Z'_d(n)) \le  \frac{n}{5.1}$. From the analysis of the previous subsection on the closeness of the distributions $\frac{h'^p}{H'}$ and $\frac{h}{H}$, it is clear that we have $\sum_{p>0} p^2\frac{h^p}{H}< \frac{n}{10}$. 
    Then we have 
    \[
    T(\eps n) = \sum_{p > \eps n}ph^p < H(\eps n)^{-1}\frac{n}{10}  =  \frac{H}{10 \eps}.
    \] 

    Adding these together, we see that it suffices to prove that 
\[
\sqrt{13}H\eps\sqrt{n} + \frac{H}{10 \eps} < \frac{H\sqrt{n}}{22}.  
\]

By our choice of $\eps$, the first term is half of the RHS.  Thus the inequality is satisfied when 
\[
\frac{\sqrt{n}}{44} > \left(\frac{1}{10\eps}\right) = \frac{44\sqrt{13}}{10}
\]
which holds when $n> 500000$, as desired. 

\end{proof}

\printbibliography
\end{document}